\documentclass[12pt]{amsart}
\usepackage{amsmath}
\usepackage{amsthm}
\usepackage{amssymb}
\usepackage{hyperref}

\newcommand{\sv}[1]{}
\newcommand{\lv}[1]{#1}
\usepackage{etoolbox}

\newcommand{\toappendix}[1]{#1}
\lv{
\usepackage[foot]{amsaddr}
\usepackage{amsfonts}
\usepackage{graphicx}
\usepackage{tabularx}
\usepackage{array}
\usepackage[usenames,dvipsnames]{color}
\usepackage{comment}
\usepackage{xcolor}
}

\sv{
\usepackage{a4wide}
}
\lv{
\usepackage{fullpage}
}

\DeclareMathOperator{\con}{\nleftrightarrow}
\DeclareMathOperator{\cc}{cc}
\newcommand{\chicon}{\chi_{\con}}

\newtheorem{theorem}{Theorem}[section]

\newtheorem{observation}[theorem]{Observation}

\newtheorem{question}[theorem]{Question}
\newtheorem{claim}[theorem]{Claim}

\newtheorem{corollary}[theorem]{Corollary}
\theoremstyle{definition}
\newtheorem{definition}[theorem]{Definition}

\title{Single-conflict colorings of degenerate graphs}
\sv{\author{
Peter Bradshaw\thanks{Department of Mathematics, Simon Fraser University, Burnaby, BC, Canada \& Department of Mathematics, University of Illinois Urbana-Champaign, Urbana, Illinois, USA. E-mail: {\tt pb38@illinois.edu}.}
\and
Tom\'a\v{s} Masa\v{r}\'ik\thanks{Department of Mathematics, Simon Fraser University, Burnaby, BC, Canada \& Faculty of Mathematics, Informatics and Mechanics, University of Warsaw, Poland. E-mail: {\tt masarik@mimuw.edu.pl}. T.M.~completed a part of this work while being supported by a postdoctoral fellowship at the Simon Fraser University through NSERC grants R611450 and R611368.
He did a part of this work supported by project BOBR that has received funding from the European Research Council (ERC) under the European Union’s Horizon 2020 research and innovation programme (grant agreement No 948057).}~\thanks{The full preprint version of this paper is available on arXiv~\cite{OurArxiv}.}
}
}

\lv{
\author{Peter Bradshaw}
\email{pabradsh@sfu.ca}
\author{Tom\'a\v{s} Masa\v{r}\'ik}
\address{Department of Mathematics, Simon Fraser University, Burnaby, BC, Canada \\\& Faculty of Mathematics, Informatics and Mechanics, University of Warsaw, Poland}
\email{masarik@mimuw.edu.pl}
\thanks{T.M.~completed a part of this work while being supported by a postdoctoral fellowship at the Simon Fraser
University through NSERC grants R611450 and R611368.
He did a part of this work supported by project BOBR that has received funding from the European Research Council (ERC) under the European Union’s Horizon 2020 research and innovation programme (grant agreement No 948057).
}\thanks{The extended abstract of this paper will appear at the EUROCOMB 2023 conference~\cite{OurEC}.}
}

\sv{
\makeatletter
}

\sv{
\newcommand{\shorttitle}{\@title}

\pagestyle{myheadings}
\markright{\shorttitle}
\def\@maketitle{%
  \newpage
  \begin{center}%
  \let \footnote \thanks
    {\small Proceedings of the 12th European Conference on Combinatorics, Graph Theory and Applications\\ EUROCOMB'23\\
    Prague, August 28 - September 1, 2023
    }
    \vskip 0.5em
    \rule{\linewidth}{0.04cm}
    \vskip 3.5em
    {\LARGE \textbf{\textsc{\@title}} \par}%
    \vskip 1.5em
    {\textbf{\textsc{(Extended abstract)}} \par}
    \vskip 2.5em%
    {\large
      \lineskip .5em%
      \begin{tabular}[t]{c}%
        \@author
      \end{tabular}\par}%
  \end{center}%
  \par
  }
\makeatother 
}

\begin{document}
\sv{
\thispagestyle{empty}
}
\maketitle

\begin{abstract}  We consider the \emph{single-conflict coloring} problem, a graph coloring problem in which each edge of a graph receives a forbidden ordered color pair. The task is to find a vertex coloring such that no two adjacent vertices receive a pair of colors forbidden at an edge joining them. We show that for any assignment of forbidden color pairs to the edges of a $d$-degenerate graph $G$ on $n$ vertices of edge-multiplicity at most $\log \log n$, $O(\sqrt{ d } \log n)$ colors are always enough to color the vertices of $G$ in a way that avoids every forbidden color pair. This answers a question of Dvořák, Esperet, Kang, and Ozeki for simple graphs (Journal of Graph Theory 2021).

  \sv{
{\noindent\small\bf DOI:}  {\tt https://doi.org/10.5817/CZ.MUNI.EUROCOMB23-000}
}
\end{abstract}

\section{Introduction}

We consider graphs without loops and possibly with parallel edges.
A \emph{coloring} of a graph $G$ is a function $\phi:V(G) \rightarrow C$ that assigns a color from some color set $C$ to each vertex of $G$. 
In this paper, we consider a special version of graph colorings known as \emph{single-conflict colorings}, defined as follows. Let $G$ be a graph, and let $C$ be a color set.
Let $f$ be a function such that for each edge $e \in E(G)$ with endpoints $u$ and $v$, 
$f$ maps the triple $(u,e,v)$ to a forbidden color pair $(c_1,c_2)$, and $f$ maps the triple $(v,e,u)$ to the reverse forbidden color pair $(c_2,c_1)$.
Then, we say that a (not necessarily proper) coloring $\phi:V(G) \rightarrow C$ is a single-conflict coloring with respect to $f$ and $C$ if $f(u,e,v) \neq (\phi(u), \phi(v))$ for each edge $e = uv$ of $G$. We call the image of a triple $(u,e,v)$ under $f$ a \emph{conflict}, and we call $f$ a \emph{conflict function}. If $k$ is the minimum integer for which a graph $G$ always has a single-conflict coloring for any color set $C$ of size $k$ and any conflict function $f$, then we say that $k$ is the \emph{single-conflict chromatic number} of $G$, and we write $\chicon(G) = k$.

\subsection{Background}
Single-conflict coloring was first considered by Fraigniaud, Heinrich, and Kosowski \cite{Fraigniaud},
and the notion of single-conflict chromatic number was later introduced by Dvo\v{r}\'ak, Esperet, Kang, and Ozeki \cite{DvorakConflict}. 
Furthermore, Dvo\v{r}\'ak and Postle \cite{DP} independently introduced a similar concept known as DP-coloring\footnote{Although \cite{Fraigniaud} was published before \cite{DP},the arXiv version of \cite{DP} appears approximately two months before that of \cite{Fraigniaud}.}. 
 In \cite{DvorakConflict}, the authors proved the following:
\begin{theorem}[\cite{DvorakConflict}]
\label{thm:delta}
If $G$ is a graph of maximum degree $\Delta$, then $\chicon(G) \leq \left \lceil \sqrt{e(2 \Delta - 1) } \right \rceil$.
\end{theorem} 
In fact, a stronger bound than in Theorem~\ref{thm:delta} with the leading multiplicative constant of $2$ (instead of $\sqrt{2e}$) can be shown.
  Surprisingly, the factor of $2$ is asymptotically sharp, as shown very recently by Groenland, Kaiser, Treffers, Wales~\cite{GroenlandKTW2023}.
For simple graphs, an even better coefficient of $1+o(1)$ holds, which can be derived from a result of Kang and Kelly~\cite{KangKelly}, or from an independent result Glock and Sudakov~\cite{GlockSudakov}.
These last three cited results were stated for independent transversals.
Another noteworthy result concerning the single-conflict chromatic number of graphs on surfaces was shown in \cite{DvorakConflict}.
\begin{theorem}[\cite{DvorakConflict}]
\label{thm:g}
If $G$ is a simple graph of Euler genus $g$, then $\chicon(G) = O((g+1)^{1/4} \log (g+2))$.
\end{theorem}
Furthermore, the authors of~\cite{DvorakConflict} show that a graph of average degree $\overline d$ has a single-conflict chromatic number of at least $\left \lfloor \sqrt{ \frac{\overline d}{\log \overline d} } \right \rfloor $.

The notion of a single-conflict coloring is a generalization of several graph coloring variants. Most immediately, single-conflict colorings generalize the notion of proper colorings as follows. Given a graph $G$, let $G^{(k)}$ denote the graph obtained from $G$ by replacing each edge of $G$ with $k$ parallel edges. Then, $\chi(G) \leq k$ if and only if $G^{(k)}$ has a single-conflict coloring with a set $C$ of $k$ colors when each set of $k$ parallel edges in $G^{(k)}$ is assigned $k$ distinct monochromatic conflicts.

A single-conflict coloring is also a generalization of a \emph{DP-coloring}, first introduced by Dvo\v{r}\'ak and Postle \cite{DP} under the name of \emph{correspondence coloring}. One may define a DP-coloring of a graph $G$ as a single-conflict coloring of a graph $G'$ on $V(G)$ which is obtained as follows. First, for each edge $uv \in E(G)$, select a matching  $M_{uv}$ in the complete bipartite graph $C \times C$. Then, for each edge $(c_1, c_2) \in M_{uv}$, give $G'$ an edge $e$ with endpoints $u,v$ and a conflict $f(u,e,v) = (c_1,c_2)$.
In this way, the single-conflict coloring problem can represent every instance of a DP-coloring problem.

Furthermore, a single-conflict coloring is a generalization of an earlier concept known as an \emph{adapted coloring}, introduced by Hell and Zhu \cite{HellZhu}, which is defined as follows. Given a graph $G$ with a (not necessarily proper) edge coloring $\psi$, an \emph{adapted coloring} on $G$ is a (not necessarily proper) vertex coloring $\phi$ of $G$ in which no edge $e$ is colored the same color as both of its endpoints $u$ and $v$---that is, $\neg \left ( \psi(e) = \phi(u) = \phi(v) \right )$. In other words, if $e \in E(G)$ is a $\mathtt{red}$ edge, then both endpoints of $e$ may not be colored $\mathtt{red}$, but both endpoints of $e$ may be colored, say, $\mathtt{blue}$, and the endpoints of $e$ may also be colored with two different colors. The \emph{adaptable chromatic number} of $G$, written $\chi_{ad}(G)$, is the minimum integer $m$ such that for any edge coloring of $G$ using a set $C$ of $m$ colors, there exists an adapted vertex coloring of $G$ using colors of $C$. It is easy to see that an adapted coloring on $G$ is a single-conflict coloring on $G$ when each edge $e \in E(G)$ is assigned the  monochromatic conflict $(\psi(e), \psi(e))$. Therefore, for every graph $G$, \[\chi_{ad}(G) \leq \chicon(G).\]

For graphs $G$ of maximum degree $\Delta$, Molloy and Thron~\cite{MolloyUB} show that $\chi_{ad}(G) \leq (1+o(1))\sqrt{\Delta}$.
Molloy~\cite{MolloyLB} shows furthermore that graphs $G$ with chromatic number $\chi(G)$ satisfy $\chi_{ad}(G) \geq (1+o(1))\sqrt{\chi(G)}$, implying that $\sqrt{\chi(G)}$, $\chi_{ad}(G)$, $\chicon(G)$, and $\sqrt{\Delta}$ all only differ by a constant factor
for graphs $G$ satisfying $\chi(G) = \Theta(\Delta)$. 
The parameters $\sqrt{\chi(G)}$, $\chi_{ad}(G)$, and $\chicon(G)$ often also differ by a constant factor even when $\chi(G)$ is not of the form $\Theta(\Delta)$. For instance, for graphs $G$ of maximum degree $\Delta$ without cycles of length $3$ or $4$, the parameters $\sqrt{\chi(G)}$, $\chicon(G)$, and $\chi_{ad}(G)$ are all of the form $O \left ( \sqrt{\frac{\Delta}{\log \Delta}} \right )$ \cite{ MolloySCC, MolloyTF}, and since randomly constructed $\Delta$-regular graphs of girth $5$ often have chromatic number as high as $\frac{1}{2} \frac{\Delta}{\log \Delta}$ \cite{MolloyTF}, this upper bound is often tight.

We note that adapted colorings are equivalent to the notion of \emph{cooperative colorings}, which are defined as follows. Given a family $\mathcal G = \{G_1, \dots, G_k\}$ of graphs on a common vertex set $V$, a cooperative coloring on $\mathcal G$ is defined as a family of sets $R_1, \dots, R_k \subseteq V$ such that for each $1 \leq i \leq k$, $R_i$ is an independent set of $G_i$, and $V = \bigcup _{i = 1}^k R_i$. A cooperative coloring problem may be translated into an adapted coloring problem by coloring the edges of each graph $G_i \in \mathcal G$ with the color $i$ and then considering the union of all graphs in $\mathcal G$. Overall, this gives us the following observation.

\begin{observation}
\label{obs:coop}
Given a family $\mathcal G = \{G_1, \dots, G_k\}$ of graphs on a common vertex set, the cooperative coloring problem on $\mathcal G$ is equivalent to the adapted coloring problem on the edge-colored graph $G = \bigcup_{i = 1}^k G_i$ in which each edge originally from $G_i$ is colored with the color $i$.
\end{observation}

It is straightforward to show that Theorem \ref{thm:delta} implies that a graph family $\mathcal G$ containing $k$ graphs of maximum degree $\Delta$ on a common vertex set $V$ has a cooperative coloring whenever $k \geq 2e \Delta$. In fact,
Haxell \cite{Haxell} showed earlier that it is sufficient to let $k \geq 2 \Delta$.

\subsection{Our results}
We have seen that for graphs $G$ of maximum degree $\Delta$, $\chicon(G) = O(\sqrt{\Delta})$. 
However, it is natural to ask whether we can obtain a better upper bound when $G$ has bounded degeneracy.
For example, in the related problem of cooperative coloring,
Aharoni, Berger, Chudnovsky, Havet, and Jiang~\cite{AharoniCoop} 
obtained the following improved result by considering families of $1$-degenerate graphs, i.e.~forests:
\begin{theorem}
[\cite{AharoniCoop}]
\label{thmTree}
If $\mathcal T$ is a family of forests of maximum degree $\Delta$ on a common vertex set $V$, then there exists a value $k = (1+o(1)) \log_{4/3}\Delta$ such that if $|\mathcal T| \geq k$, then $\mathcal T$ has a cooperative coloring.
\end{theorem}

One key tool used to prove Theorem  \ref{thmTree} is an application of the Lov\'asz Local Lemma in which each vertex $v \in V$ receives a random \emph{inventory} $S_v$ of colors from a color set $C$ indexing the forests in $\mathcal T$, and then a color $c$ is deleted from $S_v$ if $c$ also belongs to the inventory $S_w$ of the parent $w$ of $v$ in the forest of $\mathcal T$ indexed by $c$.
Dvořák, Esperet, Kang, and Ozeki \cite{DvorakConflict} also posed the following question, asking whether 
this upper bound can be improved for certain degenerate graphs.
\begin{question}
\label{qDegen}
Suppose $G$ is a $d$-degenerate graph on $n$ vertices. Is it true that 
$\chicon(G) = O  ( \sqrt{d} \log n )$?
\end{question}

The authors remarked that a positive answer to Question \ref{qDegen} would give an alternative proof of Theorem \ref{thm:g}.
In this paper, we will 
prove the following theorem, which shows the upper bound of $\chicon(G) = O(\sqrt{\Delta})$ can often be improved for graphs of bounded degeneracy.
\begin{theorem}
\label{thm:mu}
If $G$ is a $d$-degenerate graph with maximum degree $\Delta$ and edge-multiplicity at most $\mu$, then $$\chicon(G) \leq \left \lceil \sqrt d \cdot  2^{ \mu/2 + 2} \sqrt{\mu} \sqrt{1 + \log((d+1)\Delta) } \right \rceil.$$
\end{theorem}
Theorem \ref{thm:mu} gives a large class of $d$-degenerate graphs $G$ satisfying $\chicon(G) = O( d^{\frac{1}{2} + o(1)})$, containing in particular those $d$-degenerate simple graphs $G$ with maximum degree $\Delta = \exp(d^{o(1)})$. An upper bound of the form $\chicon(G) = O( d^{\frac{1}{2} + o(1)})$ is close to best possible for many classes of $d$-degenerate graphs $G$, since Molloy \cite{MolloyLB} shows that $d$-degenerate graphs $G$ of chromatic number $d+1$ satisfy $\chicon(G) \geq \chi_{ad}(G) \geq (1 + o(1)) \sqrt{d+1}$. 
By applying the argument used for Theorem \ref{thm:mu} to simple graphs, we also obtain the following theorem.

\begin{theorem}
\label{thm:easyintro}
If $G$ is a $d$-degenerate simple graph of maximum degree $\Delta$, then 
\[\chicon(G) \leq \left \lceil 2\sqrt{d \left [ 1 + \log ( ( d + 1) \Delta)   \right ]} \right \rceil .\]
\end{theorem}
Theorem \ref{thm:easyintro} immediately answers Question \ref{qDegen} for simple graphs and thus also implies Theorem \ref{thm:g}. In fact, the result that we will prove is slightly stronger than Theorem \ref{thm:easyintro}, and we will obtain the following corollary, which generalizes Theorem \ref{thmTree} at the expense of a constant factor.

\begin{corollary}
\label{cor:intro}
Let $\mathcal G$ be a family of $k$ graphs on a common vertex set $V$. Suppose each graph $G \in \mathcal G$ is at most $d$-degenerate and of maximum degree $\Delta$. Then, whenever
$k \geq 13(1+d \log(d \Delta))$,
$\mathcal G$ has a cooperative coloring.
\end{corollary}

\lv{The proof of Corollary \ref{cor:intro} appears in Section \ref{sec:unique}.}

One natural question is whether the logarithmic factors are necessary in these upper bounds. While we are unable to answer these questions exactly, we note that an upper bound of less than $d+1$ is unachievable, as Kostochka and Zhu \cite{Kostochka} give examples of $d$-degenerate graphs $G$ that satisfy $\chi_{ad}(G) = d+1$. Additionally, Question \ref{qDegen} remains open for graphs of large edge-multiplicity. It is worth noting
that the Question \ref{qDegen} has a negative answer if $G$ is allowed to have edge-multiplicity $\mu  = \omega( \log^2 n)$. Indeed, consider the graph obtained from $K_{1,n-1}$ by replacing each edge with $\mu$ parallel edges.
Then, $G$ is a $\mu$-degenerate graph on $n$ vertices, and it is straightforward to show that $\chicon(G) = \mu + 1 = \omega(\sqrt{\mu} \log n)$.

\section{Uniquely restrictive conflicts}
\label{sec:unique}
\sv{
It is well known that an oriented graph with maximum out-degree $d$ is $d$-degenerate. Therefore, rather than working directly with $d$-degenerate graphs, we will consider the larger class of oriented graphs of maximum out-degree $d$.
Given an oriented graph $G$, we write $A(G)$ for the set of arcs of $G$. For a vertex $v \in V(G)$, we write $A^+(v)$ for the set of arcs outgoing from $v$, and we write $A^-(v)$ for the set of arcs incoming to $v$. Given an arc $e = uv$ in an oriented graph $G$, and given a conflict function $f$ on $G$, we will often write $f(e) = f(u,e,v)$.
}
\lv{
In this section and the next, we will consider graphs of bounded degeneracy. Suppose that $G$ is a $d$-degenerate graph on $n$ vertices with a linear vertex ordering $v_1, \dots, v_n$ in which each vertex $v_j$ has at most $d$ neighbors $v_i$ satisfying $i < j$. 
Then, we observe that there exists an orientation of $E(G)$ of maximum out-degree $d$ obtained by giving each edge $v_i v_j$ with $i < j$ an orientation $(v_j, v_i)$. Therefore, in order to prove a statement for $d$-degenerate graphs, it is enough to prove the statement for oriented graphs of maximum out-degree $d$.
Given an oriented graph $G$, we write $A(G)$ for the set of arcs of $G$. For a vertex $v \in V(G)$, we write $A^+(v)$ for the set of arcs outgoing from $v$, and we write $A^-(v)$ for the set of arcs incoming to $v$. 
We also write $N^+(v)$ and $N^-(v)$ for the sets of out-neighbors and in-neighbors of $v$, respectively.
Given an arc $e = uv$ in an oriented graph $G$, and given a conflict function $f$ on $G$, we will often write $f(e) = f(u,e,v)$.
}

Consider a color set $C$ and an oriented graph $G$ with a conflict function $f$. First, given a vertex $v \in V(G)$ and an arc $e \in A(G)$ containing $v$ and second endpoint $u$, we say that the \emph{$(v,e)$ conflict color} is the first color appearing in the ordered pair $f(v,e,u)$. We write $\cc(v,e)$ for the $(v,e)$ conflict color. Then, we have the following definition.
\begin{definition}
Let $w \in V(G)$. Suppose that for each parallel arc pair $e_1, e_2 \in A^-(w)$ satisfying $\cc(w,e_1) = \cc(w,e_2)$, it holds that $\cc(v,e_1) = \cc(v,e_2)$, where $v$ is the second endpoint of $e_1$ and $e_2$. Then, we say $f$ is \emph{uniquely restrictive} at $w$. Furthermore, if $f$ is uniquely restrictive at each $w \in V(G)$, then we simply say that $f$ is \emph{uniquely restrictive}.
\end{definition}

An informal way of describing unique restrictiveness would be to say that if we color a vertex $w \in V(G)$ with some color, say \texttt{red}, then we only want this choice of \texttt{red} at $w$ to contribute to the exclusion of at most one color possibility at each in-neighbor of $w$. We note that unique restrictiveness is a rather natural idea, as the conflict functions that represent adapted coloring and proper coloring problems are uniquely restrictive; indeed, in both of these settings, choosing the color \texttt{red} at a vertex $v$ can only contribute to the exclusion of the color \texttt{red} at neighbors of $v$. Furthermore, DP-coloring problems always give uniquely restrictive conflict functions when represented as single-conflict coloring problems since the conflicts between any two vertices form a matching in $C \times C$. 

We will also use the following form of the Lov\'asz Local Lemma.

\begin{theorem}[\cite{LLL}]\label{thm:LLL}
Let $\mathcal B$ be a set of bad events. Suppose that each event $B \in \mathcal B$ occurs with probability at most $p$, and suppose further that each event $B \in \mathcal B$ is mutually independent with all but at most $d$ other events $B' \in \mathcal B$. If $ep(d+1) \leq 1$,
then with positive probability, no event in $\mathcal B$ occurs.
\end{theorem}

With these preliminaries in place, we have the following theorem, which gives an upper bound on the number of colors needed for a single-conflict coloring of a $d$-degenerate graph whose conflict function is uniquely restrictive. Since any conflict function on a simple graph is uniquely restrictive, the following theorem implies Theorem \ref{thm:easyintro} and hence gives an affirmative answer to Question \ref{qDegen}.
Our main tool for this theorem will be the application of the Lov\'asz Local Lemma used by Aharoni, Berger, Chudnovsky, Havet, and Jiang~\cite{AharoniCoop} in which each vertex receives a random inventory of colors.

\begin{theorem}
\label{thmOrientation}
Let $G$ be an oriented graph of maximum degree $\Delta$ with a maximum out-degree of at most $d$.
Let $C$ be a set of $k$ colors, and let each arc $e \in A(G)$ have an associated conflict $f(e) \in C^2$. If $f$ is uniquely restrictive, and if 
$k \geq   2\sqrt{d \left [ 1 + \log ( ( d + 1) \Delta)   \right ] }$,
then $G$ has a single-conflict coloring with respect to $f$ and $C$.
\end{theorem}
\begin{proof}
First, we note that since every subgraph of $G$ has an average degree of at most $2d$, $G$ is $(2d)$-degenerate and hence has a single-conflict coloring whenever $k \geq 2d + 1$. Therefore, we may assume in our proof that $k \leq 2d$.

First, for each vertex $v \in V(G)$, we define a color inventory $S_v$, and for each color $c \in C$, we add $c$ to $S_v$ independently with probability $p =  \frac{k}{2  d} \leq 1$. 
Next, we let $S'_v$ be a subset of $S_v$ obtained as follows. 
First, we initialize $S_v' = S_v$. Then, for each vertex $v \in V(G)$, we consider each outgoing arc $e$ of $v$, and we write $e = (v,w)$. If, for some color $c \in S_v$, we have
\[f(e) \in \{(c,c'): c' \in S_w\},\]
then we delete $c$ from $S_v'$. In other words, if the color $c$ at $v$ contributes to the forbidden pair $f(v,w) = (c,c')$ of an outgoing arc $(v,w) \in A^+(v)$,
and if $c' \in S_w$, then we delete $c$ from $S_v'$. Then, for each vertex $v \in V(G)$, we let $B_v$ denote the bad event that $S_v'$ is empty. 
We observe that if no bad event occurs, then we may arbitrarily color each vertex $v$ with a color from $S_v'$ to obtain a single-conflict coloring of $G$.
Indeed, if some arc $(v,w)$ is colored with a forbidden pair $(c,c')$ where $c \in S_v'$ and $c' \in S_w'$, then it must follow that $c$ was actually deleted while constructing $S_v'$, a contradiction.
 
Now, given a vertex $v \in V(G)$, we estimate $\Pr(B_v)$, that is, the probability that $S_v' = \emptyset$.
For a given color $c \in C$, we write $b_c$ for the number of arcs $e \in A^+(v)$ for which $c = \cc(v,e)$. 
If $c$ does not belong to $S'_v$, then either $c \not \in S_v$, or $c$ belongs to $S_v$ and was deleted when constructing $S'_v$. 
The probability that $c \not \in S_v$ is equal to $1 - p$, and the probability that $c \in S_v$ and $c \not \in S'_v$ is at most $b_c p^2$.
Therefore, the total probability that $c \not \in S'_v$ is at most $1 - p +  b_{c}p^2$.
Furthermore, since $f$ is uniquely restrictive,  the probabilities of any two given colors being absent from $S_v'$ are independent.
Therefore, the probability of the bad event $B_v$ is at most
\[\prod_{c \in C} \left (1 -\left  (p - b_{c}p^2 \right ) \right ) < \exp \left ( -\sum_{c \in C}            \left  (p -b_{c}p^2 \right )          \right )
= \exp \left (  -pk + p^2   \sum_{c \in C}     b_{c}          \right ) = \exp \left (  -pk + p^2  d       \right ).
\]
Substituting $p =  \frac{k}{2 d} $, we see that 
$\Pr (B_v) < \exp \left (  -\frac{k^2}{4d}    \right ).$
Furthermore, as the bad event $B_v$ involves at most $d+1$ vertices (namely $v$ and at most $d$ out-neighbors of $v$), each of maximum degree $\Delta$, $B_v$ is mutually independent with all but fewer than $(d+1)  \Delta$ other bad events. Therefore, using the Lov\'asz Local Lemma (Theorem \ref{thm:LLL}), we see that $G$ receives a single-conflict coloring with positive probability as long as \[e (d+1) \Delta \exp \left (  -\frac{k^2}{4d}    \right ) \leq 1.\]
 This inequality holds whenever
 \sv{ $k \geq  2 \sqrt{d [1 + \log((d+1)\Delta)]},$}
 \lv{$$k \geq  2 \sqrt{d [1 + \log((d+1)\Delta)]},$$}
 which completes the proof.
\end{proof}

Using Theorem \ref{thmOrientation}, we can prove Corollary \ref{cor:intro}, which gives an upper bound on the number of colors needed for a cooperative coloring of a family of degenerate graphs.
\sv{The proof is available in the full version on arXiv~\cite{OurArxiv}.}
\toappendix{
\begin{proof}[Proof of Corollary \ref{cor:intro}]
If $\Delta \leq 28$, then as $13(1+\log(\Delta)) \ge 2\Delta$, a theorem of Haxell \cite{Haxell} gives us the result. 
  Hence, we assume that $\Delta > 28$. Furthermore, if $d = 1$, then the corollary holds by Theorem \ref{thmTree}, since the $1+o(1)$ coefficient in this theorem is less than $3$ for $\Delta > 10$. Hence, we assume that $d \geq 2$.
 Additionally, if $\Delta \leq 70$, then as $13(1+2\log(2\Delta))\ge 2\Delta$, a theorem of Haxell \cite{Haxell} gives us the result.
 
For $\Delta > 70$, it is enough to prove the theorem when $k =  \left \lceil    13{d(1+ \log (d \Delta))}   \right \rceil$.
By Observation \ref{obs:coop}, the graph $G = \bigcup_{H \in \mathcal G} H$ may be edge-colored in such a way that the cooperative coloring problem on $\mathcal G$ is equivalent to the adapted coloring problem on $G$. Observe that the maximum degree of $G$ is at most $k \Delta$, and $G$ has an orientation of its edges so that every vertex has an out-degree of at most $kd$.

By Theorem \ref{thmOrientation}, $G$ contains an adapted coloring as long as 
 \[k \geq  2 \sqrt{kd (1 + \log((kd+1)k\Delta))},\]
 or stronger, as long as $k \geq 4 d (1 + \log(2 k^2d\Delta))$. Hence, we aim to show that this inequality holds.
 
We observe that 
\[k \geq 13 d (1 + \log(d \Delta))> 4 d(1 + \log (  d^{3.25} \Delta ^{3.25})).\]
Since $d \geq 2$ and $\Delta > 70$, we know that $k^2 =  \left \lceil    13{d(1+ \log (d \Delta))}   \right \rceil ^2 < \frac{1}{2}d^{2.25}\Delta^{2.25}$, so the previous line tells us that 
\[k > 4d(1 + \log (2k^2 d \Delta)),\]
which is exactly what we needed to show.
\end{proof}
}

If $G$ does not have parallel edges, then any conflict function $f:E(G) \rightarrow C^2$ must be uniquely restrictive. Then, Theorem \ref{thmOrientation} tells us that 
$\chicon(G) \leq 2 \left \lceil \sqrt{d \left (1 + \log ( (d+1) \Delta) \right ) }  \right \rceil,$
which gives an affirmative answer to Question \ref{qDegen} for simple graphs.

\section{General conflicts}
Given an oriented graph $G$ with a conflict function $f$, we define the \emph{restrictiveness} of $f$ at $v$ as the maximum value $r_v$ for which there exists an $r_v$-tuple of parallel arcs in $A^+(v)$ whose conflicts form a set
\sv{$\{({c_1},c^*), (c_2, c^*), \dots, (c_{r_v}, c^*)\},$}
\lv{$$\{({c_1},c^*), (c_2, c^*), \dots, (c_{r_v}, c^*)\},$$}%
where the first entry in each conflict corresponds to $v$, where $c^* \in C$ is any single color, and where $c_1, \dots, c_{r_v}$ are all distinct colors. 
Then, we say that the \emph{restrictiveness} of $f$ is the maximum restrictiveness $r_v$ of $f$ at $v$, taken over all vertices $v \in V(G)$.
The restrictiveness $r$ of a uniquely restrictive conflict function satisfies $r = 1$. If $f$ is a conflict function on a graph $G$ of edge-multiplicity at most $\mu$, then the restrictiveness $r$ of $f$ satisfies $r \leq \mu$.

In this section, we will show in the following theorem that we can also find an upper bound on the number of colors needed for a single-conflict coloring given a conflict function whose restrictiveness $r$ is known but may be greater than $1$. Since $r \leq \mu$ for any graph $G$ with edge multiplicity at most $\mu$, the following theorem
also proves Theorem \ref{thm:mu}, giving an upper bound for $\chicon(G)$ of $d$-degenerate graphs $G$ with small edge-multiplicity.

\begin{theorem}\label{thm:general}
Let $G$ be an oriented graph of maximum degree $\Delta$ with a maximum out-degree of at most $d$. Let $C$ be a set of $k$ colors, and let each arc $e \in A(G)$ have an associated conflict $f(e)$. If the restrictiveness of $f$ is at most $r$, and if 
\[k \geq    \sqrt d \cdot  2^{ r/2 + 2} \sqrt{r} \sqrt{1 + \log((d+1)\Delta) },\]
then $G$ has a single-conflict coloring with respect to $f$ and $C$.
\end{theorem}

\toappendix{
  \begin{proof}[Proof of Theorem~\ref{thm:general}]
We will use a similar probabilistic method as in Theorem \ref{thmOrientation}, but we will need to work much harder to show that the Lov\'asz Local Lemma still applies. We write $C = \{1, \dots, k\}$ for our color set. We define our color inventories $S_v$ and $S_v'$, as well as our bad events $B_v$, in the same way as Theorem \ref{thmOrientation}, but this time we will use the probability value
\begin{equation*}
p = \frac{k}{   2^{r + 3} r d }.
\end{equation*}
Again, we can assume that $k \leq 2d$, so we assume that $p \leq \frac{1}{4}$.
 
Now, given a vertex $v \in V(G)$, we estimate $\Pr(B_v)$, that is, the probability that $S_v' = \emptyset$. 
For each $1 \leq t \leq k$ and each subset
$$C' = \{{c_1}, \dots, {c_t}\} \subseteq C,$$
we write $b(C')$ to denote the number of $t$-tuples of parallel outgoing arcs from $v$ whose conflicts form a set
$$\{({c_1},c^*), (c_2, c^*), \dots, (c_t, c^*)\},$$
where the first entry in each conflict corresponds to $v$, and where $c^* \in C$ is any single color. In other words, $b( C')$ denotes the number of ways that $ {c_1}, \dots,  {c_t}$ might all be simultaneously deleted during the construction of $S_v'$ due to a single color $c^*$ in the inventory of an out-neighbor of $v$. We will also write $b(c_1, \dots, c_t) = b(C')$.

Now, suppose that the bad event $B_v$ occurs, so that $S'_v$ is empty. It must have happened that for some value $0 \leq t \leq k$, exactly $t$ colors were added to $S_v$, and then those $t$ colors were all deleted while defining $S'_v$. We write $P_t$ for the probability that $|S_v| = t$ and $S_v' = \emptyset$, and we observe that
$$\Pr(B_v) \leq P_0 + P_1 + \dots + P_k.$$
We aim to estimate each value $P_t$.
In order to make our proof as easy to understand as possible, we will start with small values of $t$.

First, when $t = 0$, it is straightforward to calculate that $P_0 = (1-p)^k$.

Next, when $t = 1$, the probability that $S_v = \{c\}$ for a given color $c \in C$ is equal to $(1-p)^{k-1}p$. Then, the probability that $c$ is subsequently deleted from $S'_v$ is at most $p\cdot b(c)$. Therefore,
the probability that $S_v = \{c\}$ and $S_v' = \emptyset$ is at most $(1-p)^{k-1}p^2 b(c)$. Summing over all values $c \in C$, 
$$P_1 \leq (1-p)^{k-1}p^2 \cdot\left ( b(1) + b(2) + \dots b(k) \right ).$$

Next, when $t = 2$, the probability that $S_v = \{c,c'\}$ for two given colors $c, c' \in C$ is equal to $(1-p)^{k-2}p^2$. 
Then, if $S_v' = \emptyset$,
there are two possible reasons that $c$ and $c'$ were both deleted while constructing $S_v'$.
First, it is possible that there is an out-neighbor $w \in N^+(v)$ that can be reached from $v$ by two distinct outgoing arcs $e, e' \in A^+(v)$ satisfying $f(e) = (c,c^*)$ and $f(e') = (c',c^*)$ for some color $c^* \in C$, and that since incidentally $c^* \in S_w$, $c$ and $c'$ were both deleted while constructing $S_v'$.
Given that $S_v = \{c,c'\}$, the probability of such a subsequent event is at most $p \cdot b(c,c')$. Second, it is possible that there exist two (possibly identical) out-neighbors $w, w' \in N^+(v)$
that can be reached by two distinct arcs $e,e' \in A^+(v)$ which satisfy $f(e) = (c,c^*)$ and $f(e) = (c', \tilde{c})$, and that since incidentally $c^* \in S_w$ and $\tilde{c} \in S_{w'}$, $c$ and $c'$ were deleted while constructing $S_v'$. 
Given that $S_v = \{c,c'\}$, the probability of such a subsequent event is at most $p^2 b(c) b(c')$. 
Thus, given that $S_v = \{c,c'\}$, the probability that $S_v' = \emptyset$ is at most $(1-p)^{k-2}p^2 ( p b(c,c') + p^2 b(c) b(c'))$.
By summing over all color pairs $c,c' \in C$, we see that
\[P_2 \leq (1-p)^{k-2} p^2  \left ( p\sum_{1 \leq i < j \leq k} b(i, j) +
p^2 \sum_{1 \leq i < j \leq k} b(i)b(j) \right ).\]
Note that the first term in the upper bound of $P_2$ bounds the probability that $S_v'$ became empty because of the presence of a single color $c^* \in S_w$ for a single out-neighbor $w$ of $v$, and the second term bounds the probability that $S_v'$ became empty because of two colors $c^* \in S_w$ and $\tilde{c} \in S_{w'}$ for two (possibly identical) out-neighbors $w,w'$ of $v$.

Now, we consider a general value $1 \leq t \leq k$. The probability that $S_v = \{c_1,\dots,c_t\}$ for a given color set $\{ {c_1}, \dots,  {c_t}\}$ is $(1-p)^{k-t}p^t$. For some value $z \leq t$, suppose that $S_v'$ is empty because of the presence of $z$ 
distinct colors $x_1 \in S_{w_1}, \dots, x_z \in S_{w_z}$, for $z$ (not necessarily distinct) out-neighbors $w_1, \dots, w_z \in N^+(v)$. 
Then it must follow that there exists a partition
$C_1 \dot{\cup} \cdots \dot{\cup} C_z = S_v$
such that for each $1 \leq i \leq z$, the color $x_i \in S_{w_i}$ causes the color set $C_i$ to be deleted while constructing $S_v'$. 
We assume, without loss of generality, that the sets $C_1, \dots, C_z$ are in lexicographic order when considered as increasing sequences. 
Since the restrictiveness of $f$ is at most $r$, we may assume that $|C_i| \leq r$ for each $i$. 
We may also assume that each part in this partition is nonempty, as otherwise, we may cover this case with a smaller value of $z$. 
Therefore, given that $S_v = \{c_1, \dots, c_t\}$, the probability that $S_v' = \emptyset$ is at most 
\[(1-p)^{k-t}p^t \left (\sum_{z = 1}^t p^z \sum_{\substack{q_1 + \dots + q_z = t \\ 1 \leq q_i \leq r \textrm{ for $1 \leq i \leq z$}}} \sum_{\substack{|C_i| = q_i\textrm{ for $1 \leq i \leq z$} \\ C_1 \dot{\cup} \cdots \dot{\cup} C_z = \{c_1,\dots,c_t\}}} b(C_1)b(C_2) \dots b(C_z) \right ).
\]
Summing over all subsets $\{c_1, \dots, c_t\} \subseteq C$ of size $t$, we may bound $P_t$ as follows:
$$P_t \leq (1-p)^{k-t}p^t \left (\sum_{z = 1}^t p^z \sum_{\substack{q_1 + \dots + q_z = t \\ 1 \leq q_i \leq r \textrm{ for $1 \leq i \leq z$}}} \sum_{\substack{|C_i| = q_i\textrm{ for $1 \leq i \leq z$} \\ \textrm{$C_i$ pairwise disjoint}}} b(C_1)b(C_2) \dots b(C_z) \right ).$$

We will make the following notational simplification. If we have sets $A_1, \dots, A_m$ that are pairwise disjoint and such that $|A_i| = a_i$ for $1 \leq i \leq m$, then we write $[A_1, \dots, A_m] = (a_1, \dots, a_m)$. Also, if $A$ satisfies $A \subseteq \{1, \dots, k\}$, $j \in A$, and $A \cap \{1, \dots, j-1\} = \emptyset$, then we write $\ell(A) = j$.
For each fixed integer partition $q_1 + \dots+ q_z = t$, we make the following claim.

\begin{claim}
\label{claimDisgusting}
\[\sum_{[C_1, \dots, C_z] = (q_1, \dots, q_z)} b(C_1)b(C_2) \dots b(C_z) \leq 2^{z r} \sum_{1 \leq {i_1} < i_2 < \dots < {i_z} \leq k}b({i_1})b({i_2}) \dots b({i_z}).\]
\end{claim}
\noindent  \textit{Proof of Claim~\ref{claimDisgusting}.}
We write \[T = \sum_{[C_1, \dots, C_z] = (q_1, \dots, q_z)} b(C_1)b(C_2) \dots b(C_z)\]
for the sum that we are bounding.

We induct on $z$. For the base case, when $z = 1$, we must show that 
\[T = \sum_{|C_1| = t} b(C_1) \leq 2^{r} \sum_{1 \leq {i} \leq k}b(i).\]
We have 
$$T = \sum_{|C_1| = t} b(C_1) \leq \sum_{j=1}^k \sum_{C_1 \ni j} b(C_1).$$
Let $j$ be fixed. Consider an outgoing arc $e = (v,w)$ from $v$ for which $\cc(v,e) = j$, and write $f(e) = (j, c^*)$. As the restrictiveness of $f$ is at most $r$, there exist at most $2^{r}$ sets of parallel edges joining $v$ and $w$ whose conflicts have $c^*$ in the entry corresponding to $w$. It follows that 
$$ \sum_{C_1 \ni j}b(C_1) \leq 2^{r} b(j).$$
Then the base case follows immediately.

Now, consider a value $z \geq 2$.
We have
$$T =  \sum_{j = 1}^k \sum_{[C_2, \dots, C_z] = (q_2, \dots, q_z)} b(C_2) \dots b(C_z) \sum_{\substack{
|C_1| = q_1 \\ C_1 \cap C_i = \emptyset \textrm{ for $2 \leq i \leq z$} \\ \ell(C_1) = j}} b(C_1). $$ 
Since $C_2, C_3, \dots$ are disjoint with $C_1$, and since the sets $C_i$ are in lexicographic order, it is equivalent to write
$$T= \sum_{j = 1}^k \sum_{\substack{[C_2, \dots, C_z] = (q_2, \dots, q_z), \\ \ell(C_i) > j \textrm{ for }2 \leq i \leq z }} b(C_2) \dots b(C_z) \sum_{\substack{
|C_1| = q_1 \\ C_1 \cap C_i = \emptyset \textrm{ for $2 \leq i \leq z$} \\ \ell(C_1) = j}} b(C_1). $$ 

By the same argument used in the base case, for fixed $j$, we have
$$\sum_{\substack{
|C_1| = q_1 , \\ \ell(C_1) = j}}
b(C_1) \leq  \sum_{C_1 \ni j}b(C_1) \leq 2^{r} b(j).$$ 
Therefore, it follows that
$$T \leq 2^{r}  \sum_{j = 1}^k b(j)    \sum_{\substack{  [C_2, \dots, C_z] = (q_2, \dots, q_z)  ,\\  \ell(C_i) > j\textrm{ for }2 \leq i \leq z }  } b(C_2) \dots b(C_z) .$$
Now, we may apply the induction hypothesis to the sum $$\sum_{\substack{[C_2, \dots, C_z] = (q_2, \dots, q_z), \\ 
\ell(C_i) > j   \textrm{ for }2 \leq i \leq z }} b(C_2) \dots b(C_z) ,$$ 
after which we have that
\begin{eqnarray*}
T &\leq& 2^{r}  \sum_{j = 1}^k b(j)  2^{(z-1)r} \sum_{j+1 \leq {i_2} < \dots < {i_z} \leq k} b({i_2}) \dots b({i_z}) \\
 & = & 2^{z r}  \sum_{1 \leq {i_1} < \dots < {i_z} \leq k}b({i_1}) \dots b({i_z}).
\end{eqnarray*}
This completes induction and proves the claim.
\hfill$\diamondsuit$\smallskip

From now on, we will simply write 
\[ \sigma_z = \sum_{1 \leq {i_1} < \dots < {i_z} \leq k}b({i_1}) \dots b({i_z}).\]
By Claim \ref{claimDisgusting},
\[P_t \leq (1-p)^{k-t}p^t \left (\sum_{z = 1}^t (2^{r}p)^z \sum_{\substack{q_1 + \dots + q_z = t \\ 1 \leq q_i \leq  r \textrm{ for $1 \leq i \leq z$}}} \sigma_z \right ).\]
As $1 \leq q_i \leq r$ for each $i$, the number of integer partitions in the second sum is at most $r^z$. Thus we have 
$$P_t \leq (1-p)^{k-t}p^t \left (\sum_{z = 1}^t (2^{r}r p)^z \sigma_z  \right ).$$
It follows that 
\begin{eqnarray}
\notag
\Pr(B_v) &\leq& P_0 + \sum_{t = 1}^k (1-p)^{k-t}p^t \left (\sum_{z = 1}^t (2^{r}r p)^z \sigma_z \right ) \\
\notag
 & = &  P_0 + \sum_{t = 1}^k \sum_{z = 1}^t  (1-p)^{k-t}p^t \left ( (2^{r}r p)^z \sigma_z \right ) \\
 \notag
  & = &  P_0 + \sum_{z = 1}^k \sum_{t = z}^k  (1-p)^{k-t}p^t \left ( (2^{r}r p)^z \sigma_z \right ) \\
  \notag
   & < &  P_0 + (1-p)^k\sum_{z = 1}^k \left ( (2^{r} r p)^z   \sigma_z \right ) \sum_{t = z}^{\infty}  \left (\frac{p}{1-p} \right )^t 
\end{eqnarray}
Since $p \leq \frac{1}{4}$, we crudely estimate
\[\sum_{t = z}^{\infty}  \left (\frac{p}{1-p} \right )^t < (2p)^z \sum_{t=0}^{\infty} (2p)^t \leq 2^{z+1}p^z.\]
Therefore,
\begin{eqnarray}
    \nonumber
   \Pr(B_v) &<& P_0 + (1-p)^k\sum_{z = 1}^k 2 (2^{r+1}r p^2 )^{z} \sigma_z\\
   \label{eqUB}
   &=&  P_0 + (1-p)^k\sum_{z = 1}^k  (2^{r/2+1} \sqrt{r} p )^{2z} \sigma_z.
\end{eqnarray}
Now, we claim that 
$$\Pr(B_v) < \prod_{i = 1}^k \left ( (1 - p + (2^{r/2+1} \sqrt{r} p)^2 b(i) \right ).$$
For convenience, we will write each factor as $\alpha + \beta_i$, where $\alpha = 1-p$ and $\beta_i = (2^{r/2+1} \sqrt{r} p)^2 b(i) $. It suffices to show that each term of $(\ref{eqUB})$ has a dominating term in the expansion of the product $\prod_{i = 1}^k (\alpha + \beta_i)$.

First, it is clear that $P_0 = \alpha^k$. Now, for each $1 \leq z \leq k$, we claim that
$$ (1-p)^k(2^{r/2+1} \sqrt{r} p)^{2z} \sigma_z$$
is bounded above by $[\alpha^{k-z}]\prod_{i = 1}^k (\alpha + \beta_i)$---that is, the sum of the terms in $\prod_{i = 1}^k (\alpha + \beta_i)$ in which $\alpha$ has an exponent of $k - z$. Indeed,
\begin{eqnarray*}
[\alpha^{k-z}]\prod_{i = 1}^k (\alpha + \beta_i) &=& \alpha^{k-z} \sum_{1 \leq i_1 < \dots < i_z \leq k} \beta_{i_1} \dots \beta_{i_z} \\
&=& \alpha^{k-z} (2^{r/2+1} \sqrt{r} p)^{2z} \sigma_z \\
&>& \alpha^{k} (2^{r/2+1} \sqrt{r} p)^{2z} \sigma_z \\
&= &(1-p)^k  (2^{r/2+1} \sqrt{r} p)^{2z} \sigma_z.
\end{eqnarray*}
Therefore, we see that 
\begin{eqnarray*}\Pr(B_v) & <& \prod_{i = 1}^k (\alpha + \beta_i) \\
 &=& \prod_{i = 1}^k \left ( (1 - p + (2^{r/2+1} \sqrt{r} p)^2 b(i) \right ) \\
 & \leq & \prod_{i = 1}^k \exp \left ( - p + (2^{r/2+1} \sqrt{r} p)^2 b(i) \right ) \\
  & = & \exp \left (-kp + (2^{r/2+1} \sqrt{r} p)^2 \sum_{i = 1}^k b(i) \right ) \\
  & = & \exp \left (-kp + (2^{r/2+1} \sqrt{r} p)^2 d   \right ).
 \end{eqnarray*}
 Now, substituting $p = \frac{k}{   2^{ r + 3} r d }$,
we have that 
 \[ \Pr(B_v) < \exp \left (-\frac{k^2}{ 2^{ r + 4} r d } \right ).\]

Now, as $B_v$ is a bad event that involves $d + 1$ vertices, and as $G$ has a maximum degree of $\Delta$, it follows that $B_v$ is dependent with fewer than $(d+1)\Delta$ other bad events. Therefore, by the Lov\'asz Local Lemma (Theorem \ref{thm:LLL}), $G$ has a single-conflict coloring as long as 
$$e \cdot (d+1)\Delta \cdot \exp \left (-\frac{k^2}{ 2^{ r + 4} r d } \right ) \leq 1,$$
which holds whenever
$$k \geq \sqrt d \cdot 2^{ r /2 + 2} \sqrt{r} \sqrt{ 1 + \log((d+1)\Delta)}.$$
This completes the proof.
\end{proof}
}

\bigskip
\paragraph{\textbf{Acknowledgements.}~~}
The authors would like to thank Ladislav Stacho and Jana Novotná for helpful discussions on the topic, as well as the referees for helpful suggestions.

\bibliographystyle{plain}
\bibliography{CooperativeBib}

\end{document}